\documentclass[11pt]{amsart}
\usepackage[left=1.25in, right=1.25in, top=1.25in, bottom=1.25in]{geometry}
\usepackage{enumitem}
\usepackage{subcaption}
\usepackage{graphicx}
\usepackage{verbatim}

\newtheorem{theorem}{Theorem}[section]
\newtheorem{corollary}[theorem]{Corollary}
\newtheorem{lemma}[theorem]{Lemma}

\newtheorem{proposition}[theorem]{Proposition}
\newtheorem{observation}[theorem]{Observation}
\newtheorem{fact}[theorem]{Fact}
\newtheorem{example}[theorem]{Example}

\newtheorem{question}[theorem]{Question}

\newtheorem{claim}{Claim}[theorem]
\newtheorem*{claim*}{Claim}

\newcommand{\ep}{\varepsilon}
\newcommand{\rr}[1]{\vec{r}({#1})}
\newcommand{\tr}[1]{\vec{tr}({#1})}
\newcommand{\trs}[1]{\vec{tr}^*({#1})}
\newcommand{\dg}[1]{\vec{\delta}({#1})}
\newcommand{\dgn}[1]{\vec{\delta}_n({#1})}
\newcommand{\dgs}[1]{\vec{\delta^*}({#1})}
\newcommand{\dgsn}[1]{\vec{\delta^*_n}({#1})}
\newcommand{\dgo}[1]{\vec{\delta^0}({#1})}

\newcommand{\tourn}[1]{\mathcal{T}_{#1}}

\newcommand{\ceiling}[1]{\left\lceil #1 \right\rceil}

\def\COMMENT#1{}
\let\COMMENT=\footnote

\title[Transitive tournament tilings in oriented graphs]{Transitive tournament tilings in oriented graphs with large minimum total degree} 
\author{Louis DeBiasio$^{1}$, Allan Lo$^{2}$, Theodore Molla$^{3}$, and Andrew Treglown$^{4}$}
\date{\today}

\begin{document}

\begin{abstract}
Let $\vec{T}_k$ be the transitive tournament on $k$ vertices.  We show that every oriented graph 
 on $n=4m$ vertices with minimum total degree $(11/12+o(1))n$ can be partitioned into vertex disjoint $\vec{T}_4$'s, and this bound is asymptotically tight.  We also improve the best known bound on the minimum total degree for partitioning oriented graphs into vertex disjoint $\vec{T}_k$'s. 
\end{abstract}

\maketitle
\noindent\footnotetext[1]{Department of Mathematics, Miami University {\tt debiasld@miamioh.edu}. Research supported in part by Simons Foundation Collaboration Grant \# 283194.}
\noindent\footnotetext[2]{School of Mathematics, University of Birmingham {\tt s.a.lo@bham.ac.uk}. Research supported in part by EPSRC, grant no. EP/P002420/1.} 
\noindent\footnotetext[3]{Department of Mathematics and Statistics, University of South Florida {\tt molla@usf.edu}. Research supported in part by NSF Grants DMS-1500121 and DMS-1800761.} 
\noindent\footnotetext[4]{University of Birmingham, United Kingdom {\tt a.c.treglown@bham.ac.uk}. Research supported  by EPSRC grant EP/M016641/1.}

\section{Introduction}
For a pair of (di)graphs  $G$ and $F$, we call a collection
of vertex disjoint copies of $F$ in $G$ an \textit{$F$-tiling}. We say
that an $F$-tiling is  \textit{perfect} if it consists of exactly $|V(G)|/|V(F)|$ copies
of $F$. Perfect $F$-tilings are sometimes referred to as \emph{perfect $F$-packings, perfect $F$-matchings} or \emph{$F$-factors}. 

The classic Hajnal--Szemer\'edi theorem~\cite{hajnal1970pcp} states that if $G$ is a graph on $n\in k \mathbb N$ vertices with minimum degree at least $(1-1/k)n$, then $G$ contains a perfect $K_k$-tiling. Moreover, there are $n$-vertex graphs with 
minimum degree  $(1-1/k)n-1$ that do not contain a  perfect $K_k$-tiling. 

  Recall that \emph{digraphs}
are graphs such that every pair of vertices has at most two edges between them, one oriented in each direction; \emph{oriented graphs} are orientations of simple graphs (so there is at most one directed edge between any pair of vertices). Note that oriented graphs are a subclass of digraphs.

Recently the study of tilings in digraphs has proven fruitful, and a number of papers have focused on developing analogs of the  Hajnal--Szemer\'edi theorem. In this setting there is more than one natural notion of degree: The \emph{minimum semidegree} $\delta ^0 (G)$ of a digraph $G$ is the minimum of its minimum outdegree $\delta ^+ (G)$ and its
minimum indegree $\delta ^- (G)$.  The \emph{minimum total degree} $\delta (G)$ of $G$ 
is  the minimum number of edges incident to a vertex in $G$. 
Thus,  
for oriented graphs $G$, $0\leq 2\delta ^0(G) \leq \delta (G) \leq n-1$.
 When there is no possibility of confusion, we often refer to the minimum total degree as the
minimum degree. 

Let $\vec{T}_k$ denote the transitive tournament on $k$ vertices and $C_3$ denote the cyclic triangle. In~\cite{MHSz} it was proven that every digraph on $n\in k \mathbb N$ vertices with minimum total degree at least $2(1-1/k)n-1$ contains  a perfect $\vec{T}_k$-tiling.
This degree condition is best possible, and the result implies the original Hajnal--Szemer\'edi theorem. A minimum semidegree version of the Hajnal--Szemer\'edi theorem was proven in~\cite{treglown2015directed} for large digraphs; this result considers
perfect $T$-tilings for any fixed tournament $T$. Finally, Czygrinow, DeBiasio, Molla and Treglown~\cite{CDMT} gave a general result which, together with a result of Wang~\cite{wang} determines the minimum total degree threshold for perfect $T$-tilings in a digraph for any tournament $T$.

For oriented graphs, the situation is much more difficult. Firstly notice that one can have arbitrarily large minimum total degree and still avoid even a single copy of an oriented graph. Indeed, a transitive tournament $G$ on $n$ vertices has 
$\delta (G)=n-1$ but contains no oriented graph with a directed cycle. Further, there are $n$-vertex tournaments (i.e. complete oriented graphs) with minimum semidegree at least $(n-4)/2$ (i.e. almost as large as possible)
 that do not contain a perfect $C_3$-tiling (see~\cite{keevs, theo}). Note though that Keevash and Sudakov~\cite{keevs} did prove that there exists a $c>0$ so that 
every sufficiently large oriented graph with minimum semidegree at least $(1/2-c)n$ contains a $C_3$-tiling covering all but at most $3$ vertices.
Additionally, 
Li and Molla~\cite{theo} recently proved that if $n$ is a sufficiently large odd multiple of $3$, every regular tournament on $n$ vertices has a perfect $C_3$-tiling, thereby verifying a conjecture of Cuckler~\cite{cuck} and Yuster~\cite{yustersurvey}.

More is known for the perfect $\vec{T}_k$-tiling problem in oriented graphs, though  understanding the general behaviour of the minimum degree threshold remains a significant challenge.
Yuster~\cite{yuster2003tiling} observed that if $G$ is an oriented graph on $n\in 3 \mathbb N$ vertices with minimum total degree at least $5n/6$, then $G$ has a perfect $\vec{T}_3$-tiling.   Furthermore, this bound is best possible. 
Balogh, Lo and Molla~\cite{balogh2017transitive} later proved an analogous result for the minimum semidegree threshold.

Yuster~\cite{yuster2003tiling} gave a bound on the total degree threshold for \emph{nearly perfect} tiling with $\vec{T}_k$.  That is if $G$ is an oriented graph on $n$ vertices with minimum total degree at least $\left(1 - 2^{-(k + \log k)}\right)n$, then $G$ has vertex disjoint copies of $\vec{T}_k$ covering all but $o(n)$ vertices.\footnote{Here and elsewhere $\log$ has base $2$.}
  Yuster also showed that if $G$ is an oriented graph on $n\in k \mathbb N$ vertices with minimum total degree at least $(1-4^{-k})n$, then $G$ has a perfect $\vec{T}_k$-tiling.  

Our main result is to asymptotically determine the minimum total degree threshold for perfect $\vec{T}_4$-tiling.  
  
\begin{theorem}\label{thm:T4tiling}
For all $\ep>0$, there exists $n_0$ such that if $G$ is an oriented graph on $n \geq n_0$ vertices, $n$ is divisible by $4$, and $\delta(G)\geq \left(\frac{11}{12}+\ep\right)n$, then $G$ has a perfect $\vec{T}_4$-tiling.  
Furthermore, for every $n$ divisible by $4$, there exists an oriented graph $G$ on $n$ vertices with
$\delta(G) = \ceiling{\frac{11n}{12}} - 1$ such that $G$ does not contain a perfect $\vec{T}_4$-tiling.
\end{theorem}

Moreover, we improve the general bounds on the minimum total degree threshold for perfect $\vec{T}_k$-tiling, showing that a slight improvement on Yuster's above mentioned bound for \emph{nearly} perfect $\vec{T}_k$-tiling   in fact  ensures that $G$ has a perfect $\vec{T}_k$-tiling.  
Let $\rr{k}$ be the smallest integer $n$ such that every tournament on $n$ vertices contains a copy of $\vec{T}_k$.

\begin{theorem}\label{inone}
  For every $k \ge 4$ and $\ep > 0$, there exists $n_0$ such that when $n \ge n_0$ and $n$ is divisible by $k$
  the following holds.
  If $G$ is an oriented graph on $n$ vertices and 
  \begin{equation*}
    \delta(G) \ge \left(1 - \frac{1}{k (2 \rr{k-1} - k + 1)} + \ep\right)n  ,
  \end{equation*}
  then $G$ contains a perfect $\vec{T}_k$-tiling. In particular, $ \delta(G) \ge \left(1 - 2^{-(k + \log k)}+\ep \right)n$ suffices here.
\end{theorem}

Roughly, we obtain both of our results by splitting the problem into two parts: determining the minimum degree threshold for ``fractional $\vec{T}_k$-tiling'' 
(which is related to ``nearly perfect $\vec{T}_k$-tiling'')
and determining the minimum degree threshold for ``$\vec{T}_k$-absorbing''. When $k=4$, we are able to determine these two thresholds exactly, which is why we obtain an asymptotically tight bound in that case.

As discussed in the following section, one can obtain a bound for the minimum degree threshold for perfect $\vec{T}_k$-tilings via an application of the Hajnal--Szemer\'edi theorem. Indeed, this is where Yuster's aforementioned bounds came from.
However, the bound in Theorem~\ref{thm:T4tiling} is lower than that obtained via  the Hajnal--Szemer\'edi theorem, demonstrating the problem in the oriented graph setting is genuinely different. 
In order to discuss  more precisely where our bounds come from, we must first discuss their connection to some more  parameters in the next two sections.

In Section~\ref{sec:lin} we give a minimum degree condition that ensures an oriented graph has a perfect fractional $\vec{T}_k$-tiling (and thus a nearly perfect $\vec{T}_k$-tiling); see Theorem~\ref{mainobs}. 
This theorem will be applied in both the proof of Theorem~\ref{thm:T4tiling} and Theorem~\ref{inone}.
In Section~\ref{sec:abs} we introduce an absorbing result which, combined with our results from Section~\ref{sec:lin}, yields Theorem~\ref{inone}.
Theorem~\ref{thm:T4tiling} is then proved in Section~\ref{sec:T4}. We finish the paper with some concluding remarks and open questions.

\section{Oriented Ramsey numbers and perfect tilings}\label{sec:orn}

Recall $\rr{k}$ is the smallest integer $n$ such that every tournament on $n$ vertices contains a  copy of $\vec{T}_k$.  Erd\H{o}s and Moser \cite{EM} proved that $2^{(1/2+o(1))k}\leq \rr{k}\leq 2^{k-1}$.  The following 
result provides $\rr{k}$ for small values of $k$.  

\begin{theorem}[see \cite{sanchez1998tournaments}]
$\rr{3}=4$, $\rr{4}=8$, $\rr{5}=14$, and $\rr{6}=28$.
\end{theorem}

One can consider Tur\'an-type questions in oriented graphs.
The following observation shows that the Tur\'an number of $\vec{T}_k$ in an oriented graph
is completely determined by $\rr{k}$ and Tur\'an's theorem.
Here we let $t(n,r)$ be the number of edges in a Tur\'an graph on $n$ vertices with $r$ parts, i.e.,
$t(n,r)$ is the number of edges in a complete $r$-partite graph on $n$ vertices with parts of size either
the ceiling or floor of $n/r$.
\begin{observation}
  The maximum number of edges in an oriented graph on $n$ vertices that does not contain
  a copy of $\vec{T}_k$ is $t(n,\rr{k}-1)$.
\end{observation}

\begin{proof}
  If $G$ is an oriented graph on $n$ vertices with more than $t(n, \rr{k}-1)$ edges, then,
  by Tur\'an's theorem,
  $G$ must contain a tournament on $\rr{k}$ vertices, which implies that
  $G$ contains a copy of $\vec{T}_k$. 

  Let $T$ be a tournament on $\rr{k}-1$ vertices that does not
  contain a $\vec{T}_k$.  Blowing-up
  each vertex of $T$ equitably to form an oriented graph on $n$ vertices, produces 
  a graph without a copy of $\vec{T}_k$ whose underlying simple graph is the Tur\'an graph on $n$ vertices
  with $\rr{k}-1$ parts.
\end{proof}

For every positive integer $n$, let $\tourn{n}$ be the collection of tournaments with vertex set $[n]$.
Let $\tr{k}$ be the smallest integer $n$ such that every $T \in \tourn{n}$ has a perfect $\vec{T}_k$-tiling.  
Note that, by induction, for $n > \tr{k}$ and divisible by $k$, 
every tournament $T \in \tourn{n}$ has a perfect $\vec{T}_k$-tiling.
A folklore result, which can be verified with a straightforward case analysis, is  that $\tr{3}=6$ (see \cite{reid1989three}), and, 
with a computer search\footnote{
Using the  \texttt{nauty} and \texttt{Traces} software package \cite{nauty},
we determined that there are $43$ tournaments on $12$ vertices which do not have a 
perfect $\vec{T}_4$-tiling. These tournaments are listed in Appendix~\ref{app:12_no_tt4}. Later, Bernard Lidick\'y \cite{bern} was able to use
this list to determine that every tournament on $16$ vertices has a perfect
$\vec{T}_4$-tiling.}, it has been shown that $\tr{4} = 16$.
Caro \cite{caro1989decomposition} proved that 
\begin{equation*}
  \tr{k} \le \rr{2k-1} + (2k-1)\rr{k} <4^k, 
\end{equation*}
but the determination of $\tr{k}$ is open for every $k \ge 5$. (See \cite[Proposition~10]{treglown2012note}
for a concise proof of Caro's upper-bound.)

For $n \ge \tr{k}/k$, let $\dgn{k}$ be the minimum integer such that every oriented graph $G$ 
on $nk$ vertices with $\delta(G) \ge \dgn{k}$ has a perfect $\vec{T}_k$-tiling, 
and define $\dg{k} := \limsup_{n} \frac{\dgn{k}}{nk}$.
The following straightforward consequence of the Hajnal--Szemer\'edi theorem, together with any bounds on $\tr{k}$ gives a bound on $\dg{k}$.  
\begin{observation}[Yuster \cite{yuster2003tiling}, Treglown \cite{treglown2012note}]
Given any $k,n \in \mathbb N$, $\dgn{k}\leq (1-\frac{1}{\tr{k}})kn$ and so
$$\dg{k}\leq 1-\frac{1}{\tr{k}}< 1-\frac{1}{4^k}.$$
\end{observation}

Since $\dg{3}=5/6=1-1/6=1-1/\tr{3}$, it was conceivable that $\dg{k}= 1-1/\tr{k}$ for all $k$.  However, Theorem~\ref{thm:T4tiling} shows that $\dg{4}=11/12$, whereas $\tr{4} = 16$; which means that Theorem~\ref{thm:T4tiling} does not follow directly from the Hajnal--Szemer\'edi theorem.

\section{Linear programming and fractional tilings}\label{sec:lin}
\subsection{Linear programming}
Let $H$ be a  $k$-uniform hypergraph. A \emph{matching} in $H$ is a collection of
vertex disjoint edges in $H$.
A \emph{fractional matching} in $H$ is a function $w: E(H) \rightarrow [0,1]$ so that for each $v \in V(H)$, $\sum _{e \ni v} w(e) \leq 1$.
The \emph{size} of the fractional matching is $\sum _{e \in E(H)} w(e)$.
By definition, the largest fractional matching in $H$ has size at most $|H|/k$ (if it has size exactly $|H|/k$ we say it is \emph{perfect}).
Define $\nu(H)$ and $\nu^*(H)$ to be the size of the largest matching and fractional matching in $H$,
respectively.

A \emph{vertex cover} for $H$ is a set of vertices in $H$ that together contain
at least one vertex from each edge in $H$.
A \emph{fractional vertex cover} for $H$ is a function $w: V(H) \rightarrow [0,1]$ so that for each $e \in E(H)$, $\sum _{v \in e} w(v) \geq 1$.
The \emph{size} of the fractional vertex cover is $\sum _{v \in V(H)} w(v)$.
Let $\tau(H)$ and $\tau^*(H)$ be the size of the smallest vertex cover and fractional vertex cover of $H$,
respectively.
By the duality theorem of linear programming, we have
\begin{equation*}
  \nu(H) \le \nu^*(H) = \tau^*(H) \le \tau(H).
\end{equation*}

For a pair of graphs or directed graphs $G$ and $F$, 
we let $H_F(G)$ be the $|V(F)|$-uniform hypergraph on the vertex set $V(G)$
in which $U \in \binom{V(G)}{|V(F)|}$ is an edge if and only if $G[U]$ contains a copy of $F$.
If $G$ is a graph we define $H_k(G) := H_{K_k}(G)$ and 
if $G$ is a directed graph we set $H_k(G) := H_{\vec{T}_k}(G)$.
We set $\nu_F(G) := \nu(H_F(G))$ and $\nu_k(G) := \nu(H_k(G))$. 
We define $\nu^*_F(G)$, $\tau^*_F(G)$, $\nu_F(G)$, $\nu^*_k(G)$, $\tau^*_k(G)$, and $\nu_k(G)$ analogously.

A \textit{fractional $F$-tiling of $G$} is a weight function on the copies of $F$ in $G$ 
that corresponds to a fractional matching in $H_F(G)$, i.e., 
for every vertex $v \in V(G)$, the sum of the weights on the copies of $F$ that contain $v$
is at most one. It is a \textit{perfect fractional $F$-tiling of $G$} 
if the sum of the weights is equal to $|V(G)|/|V(F)|$.
We call a weight function on the vertices of $G$ a \textit{fractional $F$-cover} 
if the weight function is a vertex cover of $H_F(G)$, that is,
if the sum of the weights on the vertices of every copy of $F$ in $G$ is at least one.
For both a fractional $F$-tiling of $G$ and a fractional $F$-cover of $G$,
the \textit{size} of the weight function is defined to be the sum of the weights
(i.e. analogous to the notion of the size of a fractional matching and a fractional vertex cover).

\smallskip

Let $\trs{k}$ denote the smallest integer $n$ such that for every $T \in \tourn{n}$
we have $\nu^*_k(T) = n/k$.
We clearly have that $\trs{k} \le \tr{k}$.
Also, every tournament $T$ on $n \ge \trs{k}$ vertices satisfies $\nu^*_k(T) = n/k$.  
Indeed, by induction on $n$, we may assume that $n > \trs{k}$ and, for each vertex $v \in V(T)$, there is a perfect fractional $\vec{T}_k$-tiling $w_v$ in $T\setminus \{v\}$.
Then $w := \frac1{n-1} \sum_{v \in V(T)} w_v$ is a perfect fractional $\vec{T}_k$-tiling in~$T$.

\subsection{Forcing fractional tilings and bounds on $\trs{k}$}

For every $n \ge \trs{k}$, 
define $\dgsn{k}$ to be the smallest integer such that every oriented graph on $n$ vertices with 
$\delta(G) \ge \dgsn{k}$ has a perfect fractional $\vec{T}_k$-tiling, and
let $\dgs{k} := \limsup_n \dgsn{k}/n$.
Let $\dgo{k}$ be the infimum of the set of numbers $\delta \in [0,1]$ such that for every $\gamma > 0$ there exists $n_0$
such that every oriented graph $G$ on $n \ge n_0$ vertices with $\delta(G)>\delta n$ has a $\vec{T}_k$-tiling of $G$ missing at most $\gamma n$ vertices.

Using our notation, we now rewrite (a slightly weaker\footnote{There are three differences to note.  
  First, we ignore the case $k=2$ and $k=3$ which Yuster considers.
  Second, Yuster proves that one can almost tile an oriented graph that meets the minimum degree condition with the blow-up of $\vec{T}_k$, but with the regularity lemma, this version of the theorem implies the original version.
  Third, Yuster writes the minimum degree condition in terms of the function $f^*(k)$ which is defined to be the smallest integer $m$ such that every tournament on at least $m$ vertices has the property that every vertex is contained in a copy of $\vec{T}_k$, but it is not hard to see that $f^*(k) = 2\rr{k-1}$ (see Example~\ref{ex:trsk_lower}).}
  version of)
Yuster's result (\cite[Theorem 3.1]{yuster2003tiling}). 

\begin{theorem}[Yuster \cite{yuster2003tiling}]\label{yuster_bound}
  For $k \ge 4$, $\dgo{k} \le 1 - \frac{1}{k (2 \rr{k-1} - 2) + 2} \le 1 - 2^{-(k + \log k)}$.
\end{theorem}

Later in this section  we prove the following bounds on $\dgo{k}$ and $\dgs{k}$ in terms of $\trs{k}$. 

\begin{theorem}\label{mainobs}
  $1 - \frac{1}{\trs{k} - 1} < \dgo{k} \le \dgs{k} \leq 1-\frac{1}{\trs{k}}$.
\end{theorem}
We also obtain the following  bounds on $\trs{k}$.

\begin{theorem}\label{boundy}
For all $k\geq 3$, $$\max\left\{ 2 \rr{k-1}, \frac{k}{k-2}\left(\rr{k}-2\right)\right\}\leq \trs{k}\leq k (2 \rr{k-1} - k + 1).$$
\end{theorem}
Note that the upper bound in Theorem~\ref{boundy} together with Theorem~\ref{mainobs} yields a slight strengthening of Theorem~\ref{yuster_bound}; 
they also can be combined with an absorbing result (Lemma~\ref{trivial_absorbing})  to give Theorem~\ref{inone} (see Section~\ref{sec:inone}).
Theorems~\ref{mainobs} and~\ref{boundy} will also be applied in the proof of Theorem~\ref{thm:T4tiling}.

We now prove Theorem~\ref{mainobs}.

\begin{proof}[Proof of Theorem~\ref{mainobs}]  
Let $G$ be an oriented graph on $n$ vertices with $\delta(G) \ge \left(1 - \frac{1}{\trs{k}} \right) n$.
Blow up each vertex of $G$ to a set of size~$\trs{k}$ and call the resulting oriented
graph~$G'$.
By the Hajnal--Szemer\'edi theorem, the simple graph underlying~$G'$ has a perfect $K_{\trs{k}}$-tiling. 
Note that each $K_{\trs{k}}$ has a perfect fractional $\vec{T}_k$-tiling in~$G'$. 
Hence $G'$ has a perfect fractional $\vec{T}_k$-tiling and so does~$G$.
So we have established that $\dgs{k} \leq 1 - \frac{1}{\trs{k}}$.

  Assume $\dgo{k} \le 1 - \frac{1}{\trs{k} - 1}$.
  Let $T$ be a tournament on $\trs{k} - 1$ vertices that does not have a perfect fractional $\vec{T}_k$-tiling; i.e. $\nu^*_k(T)<|T|/k$.
	Let $\gamma := \frac{|T|/k - \nu^*_k(T)}{|T|/k}$ and note that $\gamma>0$.
  For $s$ sufficiently large, blow up each of the vertices of $T$ into a set of $s$ vertices to form an oriented graph $G$
  on $n = s \cdot (\trs{k} - 1)$ vertices.
  Since $\delta(G)/n = 1 - s/n =  1 - 1/(\trs{k} - 1) \ge \dgo{k}$, and $n$ is sufficiently large, we can 
  assume that there exists a $\vec{T}_k$-tiling $\mathcal{T}$ of $G$ that covers all but at most $0.9 \gamma n$ vertices.
  Because every $\vec{T}_k$ in $G$ corresponds to a $\vec{T}_k$ in $T$, we can create a 
  fractional $\vec{T}_k$-tiling of $T$ by giving each $\vec{T}_k$ in $T$ weight equal to the number 
  of times a $\vec{T}_k$ that corresponds to it appears in $\mathcal{T}$ divided by $s$.
  This fractional $\vec{T}_k$-tiling of $T$ has size 
  \begin{equation*}
    \frac{|\mathcal{T}|}{s} \ge \frac{\left(1 - 0.9 \gamma\right)n}{ks} = \left(1 - 0.9 \gamma \right)\frac{|T|}{k} 
    > \left(1 - \gamma \right)\frac{|T|}{k} = \nu^*_k(T), 
  \end{equation*}
  a contradiction.
  So, we have established that $\dgo{k} > 1 - \frac{1}{\trs{k} - 1}$. 

  To complete the proof, we need to show that $\dgo{k} \le \dgs{k}$.  
  This can be shown by following a standard application of Szemer\'edi's regularity lemma.\footnote{It is also possible to establish this fact without appealing to the regularity lemma, 
  e.g., see \cite{alon2012large}.}
 Since the argument is standard we only sketch the proof. It suffices to show that
given any $\delta >\dgs{k}$ and any $\gamma > 0$, there exists $n_0$
such that every oriented graph $G$ on $n \ge n_0$ vertices with $\delta(G)>\delta n$ has a $\vec{T}_k$-tiling  missing at most $\gamma n$ vertices.

Let $G$ be such an oriented graph. Applying the regularity lemma one can obtain an oriented spanning subgraph $R'$ of the so-called reduced digraph $R$ of $G$ where 
$\delta(R')> \dgs{k}|R'|$. Thus, (as $R'$ is sufficiently large) $R'$ contains a perfect fractional $\vec{T}_k$-tiling. Using this fractional tiling as a framework, the counting lemma associated with the regularity lemma now ensures
$G$ contains a $\vec{T}_k$-tiling missing at most $\gamma n$ vertices.
\end{proof}


\subsection{Proof of Theorem~\ref{boundy}}

The following example gives a lower bound on $\dgo{k}$, which together with Theorem~\ref{mainobs} gives a lower bound on~$\trs{k}$.

\begin{example}\label{ex:dgok_lower}
  Let $k \ge 3$.
  For every $n\geq \rr{k}$ and $0 < \gamma < 1$, there exists an oriented graph $G$ on $n$ vertices with
		\begin{equation*}
    \delta(G) \ge \left \lfloor \left(1 - \frac{k-2}{k \left(\rr{k}-2\right)}\right)n - {\frac{2 \gamma n + k}{k (\rr{k} - 2)}} \right \rfloor,
  \end{equation*}
  such that no $\vec{T}_k$-tiling covers more than $(1 - \gamma)n$ vertices of $G$.
  In particular, this implies that $\dgo{k} \ge 1 - \frac{k-2}{k \left(\rr{k}-2\right)}$ which implies $\trs{k} \ge  \frac{k}{k-2}\left(\rr{k}-2\right)$ by Theorem~$\ref{mainobs}$.
\end{example}
\begin{proof}
  Take the largest tournament which does not contain $\vec{T}_k$; note that it has exactly $\rr{k}-1$ vertices.  
  For $\gamma > 0$, blow up one of the vertices to a set $X$ of size $\lfloor(1 - \gamma)2n/k \rfloor$ and inside the set add all possible edges (oriented arbitrarily).  Blow-up the other $\rr{k}-2$ parts to independent sets of size either the floor or ceiling of 
  \begin{equation*}
    \left(n - |X|\right) \cdot \frac{1}{\rr{k} - 2} \le 
    \frac{(k-2)n}{k \left(\rr{k}-2\right)} + \frac{2\gamma n + k}{k (\rr{k} - 2)},
  \end{equation*}
	whilst ensuring the resulting oriented graph $G$ has $n$ vertices.
  Note that every $\vec{T}_k$ must use at least $2$ vertices from $X$, so there is only space for at 
  most $(1 - \gamma)n/k$ vertex disjoint copies of $\vec{T}_k$ in $G$.  
\end{proof}

The next example gives a different lower bound on $\trs{k}$, which together with Example~\ref{ex:dgok_lower} implies the lower bound in Theorem~\ref{boundy}.

\begin{example}\label{ex:trsk_lower}
  For every $k \ge 3$, $\trs{k} \ge 2 \rr{k-1}$.
\end{example}

\begin{proof}
  To see that $\trs{k} \ge 2 \rr{k-1}$, 
  consider a tournament $T$ on $n = 2 \rr{k-1} - 1$ vertices in which there exists a vertex $u \in V(T)$
  such that $|N^+(u)| = |N^-(u)| = (n-1)/2 = \rr{k-1} - 1$; both $N^+(u)$ and $N^-(u)$ induce a tournament on
  $\rr{k-1} - 1$ vertices that does not contain a $\vec{T}_{k-1}$; all of the edges between $N^+(u)$ and $N^-(u)$
  are directed from $N^+(u)$ to $N^-(u)$.
		This ensures that $T$ does not contain a transitive tournament that contains $u$ and elements from both
		$N^+(u)$ and $N^-(u)$.
		Thus,  $u$ is not contained in a $\vec{T}_{k}$; this immediately implies 	$T$ does not have a perfect fractional $\vec{T}_k$-tiling.
\end{proof} 

To prove the upper bound of Theorem~\ref{boundy}, we first collect together some useful observations.

For a hypergraph $H$ and for every $v \in V(H)$, we let $H(v)$
be the \textit{link graph of $v$}, i.e., 
$H(v)$ is the hypergraph with vertex set $V(H)$ and edge set $\{e \setminus \{v \} : e \in E(H) \text{ and } v \in e \}$.
The following lemma is well-known.  We provide a proof for completeness.
\begin{lemma}\label{extendable}
  If $H$ is a $k$-uniform hypergraph on $n$ vertices and, 
  for every $v \in V(G)$, $\nu^*(H(v)) \ge n/k$,
  then $\nu^*(H) = n/k$.
\end{lemma}
\begin{proof}
  Suppose that $\nu^*(H(v)) \geq  n/k$
	for every $v \in V(G)$, and $\nu^*(H) < n/k$.
  In a fractional matching of $H$ of size $\nu^*(H)$, there must exist a 
  vertex $v$ in which the sum of the weights on the edges incident to $v$ is strictly less
  than $1$.
  By the complementary slackness theorem from linear programming, 
  this implies that 
  if $w$ is a fractional vertex cover of $H$ of size $\tau^*(H) = \nu^*(H)$, then $w(v) = 0$.
  This means that $w$ is a fractional vertex cover of $H(v)$, so 
  \begin{equation*}
    \nu^*(H(v)) = \tau^*(H(v)) \le \tau^*(H) = \nu^*(H) < n/k,
  \end{equation*}
  a contradiction.
\end{proof}

Let $G$ and $F$ either be a pair of graphs or a pair of directed graphs such that 
$|G| = n$ and $|F| = k$ and let $H := H_F(G)$.
For a vertex $v \in V(G)$, a weight function $w$ on the $(k - 1)$-subsets of $V(G)$
is a \textit{$v$-extendable fractional $F$-tiling of size $r$} if it corresponds
to a fractional matching of size $r$ in the hypergraph $H(v)$.
We have the following corollary to Lemma~\ref{extendable}.
\begin{corollary}\label{cor:extend1}
  Let $G$ and $F$ either be a pair of graphs or a pair of directed graphs
  such that $|G| = n$ and $|F| = k$.
  If, for every $v \in V(G)$, there exists a $v$-extendable fractional $F$-tiling of size
  at least $n/k$, then there exists a perfect fractional $F$-tiling of $G$.
\end{corollary}
\begin{proof}
  This follows from Lemma~\ref{extendable} if we consider the
  hypergraph $H_F(G)$.
\end{proof}

%

We now prove the upper bound in Theorem~\ref{boundy}.
\begin{lemma}\label{lem:trsk_upper}
  For $k \ge 3$, $\trs{k} \le k (2 \rr{k-1} - k + 1)$.
\end{lemma}
\begin{proof}
  Let $T$ be a tournament on $n := k (2 \rr{k-1} - k + 1)$ vertices.
  For an arbitrary $v \in V(T)$, we aim to prove that there exists a $v$-extendable fractional
  $\vec{T}_k$-tiling of size at least $n/k$.  
  By Corollary~\ref{cor:extend1}, this will then prove the lemma.
  To do this, we first prove the following claim.

  \begin{claim}
    If $S$ is a tournament on $s \ge \rr{k-1}$ vertices, then
    $\nu^*_{k-1}(S) \ge \frac{s - (\rr{k-1} -k + 1)}{k-1}$.
  \end{claim}
  \begin{proof} 
    Let $w$ be a fractional $\vec{T}_{k-1}$-cover of $S$ of size $\tau^*_{k-1}(S) = \nu^*_{k-1}(S)$
  and let $v_1, \dotsc, v_s$ be an ordering of $V(S)$ such that
  $w(v_1) \le w(v_2) \le \dotsm \le w(v_s)$.
  Note that $S[\{v_1, \dotsc, v_{\rr{k-1}}\}]$ contains at least one $\vec{T}_{k-1}$, so
  $\sum_{i=1}^{\rr{k-1}} w(v_i) \ge 1$
  and $w(v_{\rr{k-1}}) \ge \frac{1}{k-1}$.  Therefore, 
  \begin{equation*}
    \tau^*_{k-1}(S) = 
    \sum_{i=1}^{\rr{k-1}} w(v_i) +
    \sum_{i=\rr{k-1} + 1}^{s} w(v_i)
    \ge 1 + \frac{s - \rr{k-1}}{k-1}
    = \frac{s - (\rr{k-1} -k + 1)}{k-1}. \qedhere
  \end{equation*}
  \end{proof}

  Recall that 
  \begin{equation}\label{eq:trsk_upper}
    2\rr{k-1} - k + 1 = \frac{n}{k}. 
  \end{equation}
  Let $\nu^{*+} := \nu^*_{k-1}(T[N^+(v)])$
  and $\nu^{*-} := \nu^*_{k-1}(T[N^-(v)])$.
	Note that $v$ forms a copy of $\vec{T}_k$ with any copy of $\vec{T}_{k-1}$ in $T[N^+(v)]$ or $T[N^-(v)]$. 
	In particular, a lower bound on $\nu^{*+}+\nu^{*-}$ gives a lower bound on the size of the largest $v$-extendable fractional $\vec{T}_k$-tiling.
	
  Suppose that $d^+_T(v) \ge d^-_T(v)$.
  If $d^-_T(v) \le \rr{k-1} - 1$, then $d^+_T(v) \ge n - \rr{k-1}$. 
  So, by \eqref{eq:trsk_upper} and the claim,
  \begin{equation*}
    \nu^{*+}
    \ge \frac{d^+_T(v) - (\rr{k-1} - k + 1)}{k-1} 
    \ge \frac{n - (2\rr{k-1} - k + 1)}{k-1} 
    = \frac{n}{k}.
  \end{equation*}
  If $d^-_T(v) \ge \rr{k-1}$, 
  then by the claim, \eqref{eq:trsk_upper}, the fact that $k \ge 3$, and the fact that $d^+_T(v) + d^-_T(v)=n-1$
  we have 
  \begin{equation*}
    \nu^{*+} + \nu^{*-} \ge
    \frac{d^+_T(v) + d^-_T(v) - 2(\rr{k-1} - k + 1)}{k-1}
    \ge \frac{n - (2\rr{k-1} - k + 1)}{k-1}
    = \frac{n}{k}.
  \end{equation*}
  An analogous argument applies if $d^-_T(v) \ge d^+ _T(v)$. 
	So
  there exists a $v$-extendable fractional $\vec{T}_k$-tiling of size at least $n/k$.
\end{proof}

\subsection{Remarks}\label{remarks}
Note that Example~\ref{ex:trsk_lower} and Theorem~\ref{mainobs} together imply that
$\dgo{k} \ge 1 - \frac{1}{2\rr{k-1} - 1}$.
If it can be shown that the lower bound on $\trs{k}$ from Example~\ref{ex:dgok_lower} is also an upper bound; i.e.
$\trs{k} = \frac{k}{k-2}\left(\rr{k}-2\right)$
(which is true for $k = 3$ and $k=4$), then
we have $2 \rr{k-1} \le \frac{k}{k-2}\left(\rr{k}-2\right)$, or
\begin{equation*}
  \rr{k} \ge \frac{2(k-2)}{k} \cdot \rr{k-1} + 2,
\end{equation*}
which would imply that $\rr{k} \ge (2 - o(1))^k$ which almost matches the  Erd\H{o}s--Moser bound of $\rr{k} \le 2^{k-1}$.  In fact, even proving that $\trs{k} \leq (\sqrt{2}-c)\rr{k}$ for some absolute constant $c>0$ would improve the best known lower bound on $\rr{k}$.
It is also worthwhile to note that $\rr{k}$ provides a lower bound on the classical Ramsey number $R(k,k)$.  Indeed if $T$ is a 
tournament on $n := \rr{k} - 1$ vertices with no $\vec{T}_k$, then the graph $G$ on $V(T)$ formed by taking any 
ordering $v_1, \dotsc, v_n$ of the vertices of $T$ and, for every $1 \le i < j \le n$, placing the edge $v_iv_j$ in $G$ 
if the edge in $T$ is directed from $v_i$ to $v_j$
has neither a clique nor an independent set of size $k$.
Therefore, it is possible that a substantial improvement to the upper bound on $\trs{k}$ could give an improvement on the best known lower bound for the diagonal Ramsey numbers.

Note that when $n \ge \frac{k}{k-2}\left(\rr{k} - 2\right)$, 
\begin{equation*}
  \frac{n - (\rr{k} - 2)}{2} \ge 
  \frac{n - \frac{k-2}{k} \cdot n}{2} = \frac{n}{k}.
\end{equation*}
A way one might attempt to prove that $\trs{k} = \frac{k}{k-2} \left(\rr{k} - 2\right)$ 
would be to first prove that equality holds in the following.
\begin{example}\label{ex:trsk}
  For $k \ge 3$, if $\rr{k} \le n \le \frac{k}{k-2}\left(\rr{k} - 2\right)$, then 
  \begin{equation*}
    \min_{T \in \tourn{n}}\left\{ \nu^*(T) \right\} \le \frac{n - (\rr{k} - 2)}{2}.
  \end{equation*}
\end{example}
\begin{proof}
  Construct a tournament $T$ on $n$ vertices by starting with a tournament on $\rr{k}-1$ vertices
  that does not contain a $\vec{T}_k$ and then blow-up one of the vertices to a set $X$
  of size $n - (\rr{k} -2)$. Then place edges between all vertices in $X$ and orient them arbitrarily.
  Because every $\vec{T}_k$ has at least two vertices in $X$, 
  we can cover all of the copies of $\vec{T}_k$ in $T$ by assigning weight $1/2$
  to the vertices in $X$ and $0$ to the vertices in $V(T)\setminus X$. Therefore, 
  \begin{equation*}
    \nu^*(T) = \tau^*(T) \le \frac{|X|}{2} = \frac{n - (\rr{k} - 2)}{2}. \qedhere
  \end{equation*}
\end{proof}
Example~\ref{ex:trsk} is quite similar to Example~\ref{ex:dgok_lower}.
We have verified that, when $\rr{k} \le n \le \frac{k}{k-2}\left(\rr{k} - 2\right)$, 
equality holds in Example~\ref{ex:trsk}
when $k$ is either $3$ or $4$.  We have no evidence that equality holds when $k \ge 5$,
and in light of the discussion above, it is, if true, likely extremely challenging to prove!

\section{The absorbing method and the proof of Theorem~\ref{inone}}\label{sec:abs}
\subsection{Absorbing}
We will apply  the absorbing method of R\"odl, Ruci\'nski and Szemer\'edi (see e.g.~\cite{rrs2}).
The basic idea of the method is to prove that a randomly constructed small set can serve
as an ``absorber'', i.e., we prove that there exists a small set that has the property that
if, after removing this set from the graph,
we can almost tile what is left of the oriented graph, then, using the absorbing set,
we can extend this partial tiling into a perfect tiling over the entire original oriented graph.

To prove that our absorbing sets exist, we will use the following lemma, which follows immediately from a lemma of Lo and Markstr\"om~\cite[Lemma~1.1]{lo2015f}.
Here we write $0 < \alpha \ll \eta < 1$ to mean that $\alpha$ is chosen to be sufficiently
small compared to $\eta$ so that all constraints in the proof of the lemma hold.
\begin{lemma}\label{linking}
  For every $k \ge 3$, 
	$i \ge 1$ and $0 < \alpha \ll \eta < 1$, there exists $n_0$
  such that for every directed graph $G$ on $n \geq n_0$ vertices the following holds.
  If, for every $x,y \in V(G)$, there are at least $\eta n^{ik-1}$ sets $L \subseteq V(G)$ 
  such that $|L| = ik - 1$ and both $G[L \cup \{x\}]$ and $G[L \cup \{y\}]$ contain perfect $\vec{T}_k$-tilings,
  then there exists $A \subseteq V(G)$ such that: 
  \begin{itemize}
    \item $|A| \le \alpha n$; $|A|$ is divisible by $k$; and 
    \item for every $W \subseteq V(G) \setminus A$, such that $|W| \le \alpha^2 n$ and $|W|$ is divisible by $k$,
      we have that $G[A \cup W]$ has a perfect $\vec{T}_k$-tiling.
  \end{itemize}
\end{lemma}

Let $k\geq 3$ and $i \geq 1$. Define $\mathcal A(k,i)$ to be the set of all $ \beta>0$ with the following property:  there exists $ \eta>0$ and  $n _0 \in \mathbb N$ so that for each $n\geq n_0$, every $n$-vertex  
oriented graph $G$ with $\delta(G)\geq \beta n$, and any pair $x,y \in V(G)$, there are at least $\eta n^{ik-1}$ sets $L \subseteq V(G)$ 
  such that $|L| = ik - 1$ and both $G[L \cup \{x\}]$ and $G[L \cup \{y\}]$ contain perfect $\vec{T}_k$-tilings.
  Let $A(k,i)$ be the infinimum of  $\mathcal A(k,i)$.  
	Write $A(k):= \inf _{i\geq 1} A(k,i)$.
	We call $A(k)$ the \emph{absorbing threshold} for $\vec{T}_k$-tiling.  

        We will make use of the following simple fact.
\begin{fact}\label{fact:deg}
  For every $r, s$ and $c$ such that $1 \le s \le r$, and $|c| < 1/r$, the following holds.
  If $G$ is a graph or oriented graph on $n$ vertices and $\delta(G) \ge (\frac{r-1}{r} + c)n$, then 
  for every $U \subseteq V(G)$ such that $|U| \ge \frac{s}{r} n$ we have 
  $\delta(G[U]) \ge (\frac{s-1}{s} + c \cdot \frac{r}{s})|U|$.
\end{fact}
\begin{proof}
  Because $n \le \frac{r}{s} |U|$, we have that $\delta(G[U])$ is at least
  \begin{equation*}
      |U| - (n - \delta(G)) \ge |U| - \left(\frac{1}{r} - c\right)n \ge |U| - \left(\frac{1}{r} - c\right) \frac{r}{s}|U|=\left(\frac{s-1}{s} + c \cdot \frac{r}{s}\right)|U|. \qedhere
  \end{equation*}
\end{proof}

\begin{lemma}\label{trivial_absorbing}
For all $k\geq 3$, $A(k,1)\leq 1-\frac{1}{4\rr{k-1}-2}$.
\end{lemma}

\begin{proof}
Let $0<\eta\ll \ep \ll 1/k$, let $n$ be sufficiently large and let $G$ be an oriented graph on $n$ vertices with $\delta(G)\geq \left(1-\frac{1}{4\rr{k-1}-2}+\ep\right)n$.  Let $x,y \in V(G)$ and set $U := N(x) \cap N(y)$, 
  $r := 4\rr{k-1} - 2$, and $s := 4\rr{k-1} - 4$.
  Since $|U| \ge sn/r$, Fact~\ref{fact:deg} (with $c = \ep$) implies that 
  \begin{equation*}
    \delta(G[U]) \geq \left(\frac{s-1}{s} + \ep \cdot \frac{r}{s} \right)|U|.
  \end{equation*}
  So by supersaturation\footnote{That is, as $G$ has minimum degree significantly above the threshold for containing a tournament on $s+1$ vertices.}
	there exist at least $\eta n^{s+1}$ tournaments~$T$ on 
  $(s+1)$ vertices in $G[U]$.
  Since $s + 1 = 4\rr{k-1} - 3$, 
  by the pigeonhole principle, for every such tournament~$T$, 
  there exists a subtournament of size at least $\rr{k-1}$ 
  in one of the four sets: $N^+(x) \cap N^+(y)$,
  $N^+(x) \cap N^-(y)$, $N^-(x) \cap N^+(y)$, and $N^-(x) \cap N^-(y)$, which partition $U$.
  This, in turn, implies that there exists $L \subseteq V(T)$ such that
  $G[L \cup \{x\}]$ and $G[L \cup \{y\}]$ are $\vec{T}_k$.
  Therefore, we have at least $\eta n^{k-1}$ of the desired sets.
	
 The choice of $\ep >0$ can be made arbitrarily small, thus we obtain that $A(k,1)\leq 1-\frac{1}{4\rr{k-1}-2}$.
	
\end{proof}

\begin{lemma}\label{frac+absorb}
  For every $k \ge 3$, $i \geq 1$ and $\ep > 0$, there exists $n_0$ such that
  for every $n \ge n_0$ that is divisible by $k$ the following holds.
  If $G$ is an oriented graph on $n$ vertices and 
  \begin{equation*}
    \delta(G) \ge \max \left\{ \dgo{k} + \ep, A(k,i)+\ep \right\}n,
  \end{equation*}
  then $G$ has a perfect $\vec{T}_k$-tiling.
\end{lemma}
\begin{proof}
  Let $0 < \alpha \ll \eta \ll \ep, 1/k,1/i$.   Let $G$ be a sufficiently large oriented graph as in the statement of the lemma.

By the degree condition  we may apply Lemma~\ref{linking} 
 to get a set $A$ such that $|A| \le \alpha n$, 
  $|A|$ is divisible by $k$, 
  and, for every $W \subseteq V(G) \setminus A$ such that $|W| \le \alpha^2 n$ and
  $|W|$ is divisible by $k$, the oriented graph $G[A \cup W]$ has a perfect
  $\vec{T}_k$-tiling.
  Since $n$ is sufficiently large and $\delta(G - A) \ge (\dgo{k}+\ep /2) |G - A|$, we
  can tile $G - A$ so that if $W$ is the set of uncovered vertices, then $|W| \le \alpha^2 n$.
  Since then $G[A \cup W]$ has a perfect $\vec{T}_k$-tiling, 
  we obtain a perfect $\vec{T}_k$-tiling of $G$.
\end{proof}

\subsection{Proof of Theorem~\ref{inone}}\label{sec:inone}
With the absorbing lemma to hand, it is now straightforward to deduce Theorem~\ref{inone} from our previous results.

\begin{proof}[Proof of Theorem~\ref{inone}]
  From Theorem~\ref{mainobs} and Lemma~\ref{lem:trsk_upper} we have that
  $\dgo{k} \le 1 - \frac{1}{k (2 \rr{k-1} - k + 1)}$.
  Since $k (2 \rr{k-1} - k + 1) \ge k \cdot \rr{k-1} \ge 4 \rr{k-1} - 2$, 
  the first part of Theorem~\ref{inone} then follows from Lemmas~\ref{trivial_absorbing} and \ref{frac+absorb}.
	
	The second part of the theorem follows by the inequality in the statement of Theorem~\ref{yuster_bound}.
\end{proof}

\section{$\vec{T}_4$-tiling - Proof of Theorem~\ref{thm:T4tiling}}\label{sec:T4}

Note that $\rr{4} = 8$.
Example~\ref{ex:dgok_lower} with ($\gamma = 1/n$) implies the second part of the theorem.
For the first part of the theorem, we will show that $\dgo{4} \le \frac{11}{12}$ (Proposition~\ref{prop:frac-12-4}) and $A(4,2)\leq \frac{11}{12}$ (Corollary~\ref{cor:A(4,2)}) which together with Lemma~\ref{frac+absorb} will complete the result. 

Note that we sometimes call $\vec{T}_3$ the \textit{transitive triangle}.

\begin{proposition}\label{prop:frac-12-4}
  $\trs{4} = 12$ and $\dgo{4} =\frac{11}{12}$.
\end{proposition}

\begin{proof}
As $\rr{4} = 8$, the lower bound in Theorem~\ref{boundy} gives $\trs{4} \ge 12$ and Example \ref{ex:dgok_lower} gives $\dgo{4} \geq \frac{11}{12}$.  Thus, 
it suffices to show that $\trs{4} \le 12$ as together with Theorem~\ref{mainobs} this implies $\dgo{4} \leq  \frac{11}{12}$. 
Let $T$ be a tournament on $12$ vertices.
  It suffices, by Corollary~\ref{cor:extend1}, to show that for every $v\in V(T)$ there exists a $v$-extendable fractional $\vec{T}_4$-tiling of size at least $3$.  Recall that $\tr{3}=6$, so every tournament on $3k\geq 6$ vertices has a perfect $\vec{T}_3$-tiling.

  Let $v\in V(T)$ and suppose without loss of generality that $d^+(v)\geq d^-(v)$.  If $d^+(v)\geq 9$, then we have three disjoint $\vec{T}_3$'s in $N^+(v)$ and we are done.  If $d^-(v)\geq 4$, then since $d^+(v)\geq 6$, we have two disjoint $\vec{T}_3$'s in $N^+(v)$ and one $\vec{T}_3$ in $N^-(v)$.  So the only case left to deal with is when $d^-(v)=3$ and $d^+(v)=8$.  We would be done as before if there exists a $\vec{T}_4$ that contains $v$ and has exactly one vertex in $N^-(v)$ and two vertices in $N^+(v)$, so assume such a $\vec{T}_4$ does not exist. This implies that 
  \begin{equation}\label{oneout}
    \text{every vertex in $N^-(v)$ has at most one out-neighbor in $N^+(v)$.}
  \end{equation} 
  In this case we find a perfect $\vec{T}_4$-tiling of $T$ directly.  By \eqref{oneout}, there exists a $\vec{T}_4$, say $F$, such that $F$ has two vertices in both $N^-(v)$ and $N^+(v)$. Let $F_1$ and $F_2$ be two disjoint transitive triangles contained in $N^+(v) \setminus V(F)$ and let $u$ be the vertex in $N^-(v) \setminus V(F)$.  By \eqref{oneout}, for either $F_1$ or $F_2$, say $F_1$, we have that $\{u\} \cup F_1$ induces a $\vec{T}_4$ (since $u$ has only in-neighbors in $F_1$). Then, $F$, $T[\{u\} \cup F_1]$, and $T[\{v\} \cup F_2]$ form the desired perfect $\vec{T}_4$-tiling of the tournament $T$.
\end{proof}

To complete the proof of Theorem~\ref{thm:T4tiling}, we  show that $A(4,2) \le \frac{11}{12}$.
Let $G$ be an oriented graph on $n$ vertices and $\delta(G) \ge \left( \frac{11}{12} + \varepsilon \right) n$. 
It is sufficient to show that,
for every pair of distinct vertices $x$ and $y$ in $G$, there are at least $\Omega(n^{7})$ 
sets $L$, each of order $7$, such that both $G[L \cup \{x\}]$ and $G[L \cup \{y\}]$ contain
two disjoint copies of $\vec{T}_4$.
At a high-level, we achieve this by noting that, by the minimum total degree condition, Fact~\ref{fact:deg} and supersaturation, 
there are $\Omega(n^{11})$ tournaments on $11$ vertices in $G[N(x) \cap N(y)]$.
The following lemma, then implies that there are $\Omega(n^7)$ of the desired sets $L$.
We provide the details of this argument in our proof of Corollary~\ref{cor:A(4,2)}, 
which appears after the proof of Lemma~\ref{lemma:main-absorb}.

\begin{lemma}\label{lemma:main-absorb}
Let $G$ be an oriented graph and let $x,y\in V(G)$ and $T\subseteq V(G)\setminus \{x,y\}$. 
 If $\{x\}\cup T$ and $\{y\}\cup T$ each induce a tournament on $12$ vertices, then there exists $Z\subseteq T$ such that $|Z|=7$ and $G[\{x\} \cup Z]$ and $G[\{y\} \cup Z]$ both contain a perfect $\vec{T}_4$-tiling.  
\end{lemma}

\begin{proof}
  For clarity, we will write $u \to v$ if the edge $uv \in E(G)$ is directed from $u$ to $v$.
  Let $K$ be the tournament induced in $G$ by the vertex set $T$.
  Call $Z \subseteq T$ a \textit{linking set} if $|Z| \in \{3, 7\}$ and
  $G[\{x\} \cup Z]$ and $G[\{y\} \cup Z]$ both have a perfect $\vec{T}_4$-tiling.
  Suppose $Z$ is a linking set.
  If $|Z| = 7$, then we clearly satisfy the conclusion of the lemma.
  If $|Z| = 3$, then, because $|T \setminus Z| = 8 = \rr{4}$,
  there exists $T' \subseteq T \setminus Z$ such that $G[T']$ is a $\vec{T}_4$, so
  with $Z \cup T'$ playing the role of $Z$ we satisfy the conclusion of the lemma.
  Suppose, for a contradiction, that no linking set exists.

Let $N^{\sigma_x,\sigma_y}:=N^{\sigma_x}_T(x)\cap N^{\sigma_y}_T(y)$ for $\sigma_x, \sigma_y\in\{+,-\}$, and 
let 
\begin{equation*}
  \mathcal{P} := \{N^{+,+}, N^{+,-}, N^{-,+}, N^{-,-}\},
\end{equation*}
and note that $\mathcal{P}$ is a partition of $V(T)$.

Let $<$ be the partial order of $\mathcal{P}$ given by $N^{-,-}<N^{-,+}<N^{+,+}$ and $N^{-,-}<N^{+,-}<N^{+,+}$,
and let $uw \in E(K)$.
We say that $uw$ \textit{violates the partial order} if $u\in U$ and $w\in W$ for distinct sets $U, W \in \mathcal{P}$ 
and either
\begin{itemize}
  \item $U$ and $W$ are incomparable, or
  \item $U < W$ and $w\to u$.
\end{itemize}
  Otherwise, we say that $uw$ \textit{satisfies the partial order}.
Note that, for every edge $uw \in E(K)$,
\begin{equation}\label{iffpo}
  \text{both $xuw$ and $yuw$ are transitive triangles} \iff \text{$uw$ satisfies the partial order.}
\end{equation}

\begin{claim}\label{claim:tri}
  Every transitive triangle $abc$ in $K$ contains at least one edge that violates the partial order.
\end{claim}
\begin{proof}
  Otherwise, by \eqref{iffpo}, both $\{a,b,c,x\}$ and $\{a,b,c,y\}$ induce copies of $\vec{T}_4$,
  so $\{a, b, c\}$ is a linking set.
\end{proof}

This immediately implies the following.
\begin{claim}\label{deg1clm}
  For every pair of distinct sets $U,W \in \mathcal{P}$ that are comparable,
  the edges between $U$ and $W$ that satisfy the partial order form a matching.
\end{claim}

Because every tournament on four vertices contains a transitive triangle, and edges that violate
the partial order must intersect two sets in $\mathcal{P}$, Claim~\ref{claim:tri} implies the following.
\begin{claim}\label{cycclm}
  For every $U \in \mathcal{P}$, we have $|U| \le 3$. 
  In particular, exactly one set in $\mathcal{P}$ has order $2$
  and the other sets in $\mathcal{P}$ each have order $3$.
  Furthermore, if $U \in \mathcal{P}$ and $|U| = 3$, then $U$ induces a cyclic triangle.
\end{claim}

Without loss of generality, suppose $|N^{-,-}| + |N^{-,+}| \le |N^{+,+}| + |N^{+,-}|$,
so either $N^{-,-}$ or $N^{-,+}$ is the set in $\mathcal{P}$ of order $2$.

\begin{claim}\label{a1b1clm}
  There exists $b \in N^{-,+}$,  $c \in N^{+,-}$, and $D \subseteq N^{-,-}$ such that 
  $|D| = 2$, and, for every $d \in D$, we have $b \to d$ and $c \to d$, i.e.,
  all of the edges between $b$ and $D$ and all of the edges between $c$ and $D$
  violate the partial order.
	Moreover, $D\cup \{b,c\}$ induces a copy of $\vec{T}_4$ in $T$.
\end{claim}

\begin{proof}
  Recall that $|N^{-,-}| \in \{2,3\}$.
  If $|N^{-,-}| = 2$, then let $D := N^{-,-}$. Since $|D| < |N^{-,+}| = |N^{+,-}| = 3$,
  Claim~\ref{deg1clm} implies that there exists $b \in N^{-,+}$ and $c \in N^{+,-}$ such that,
  for every $d \in D$, we have $b \to d$ and $c \to d$.

  Suppose $|N^{-,-}| = 3$ and let $b \in N^{-,+}$. By Claim~\ref{deg1clm}, there exists $D \subseteq N^{-,-}$ such
  that $|D| = 2$ and $b \to d$ for every $d \in D$.
  Since $|D|  < |N^{+,-}|$, Claim~\ref{deg1clm} implies that there exists $c \in N^{+,-}$ such that
  $c \to d$ for every $d \in D$.
\end{proof}

\begin{figure}[t!]
  {\centering
    \begin{subfigure}{.45\textwidth}
      \includegraphics[scale=0.75]{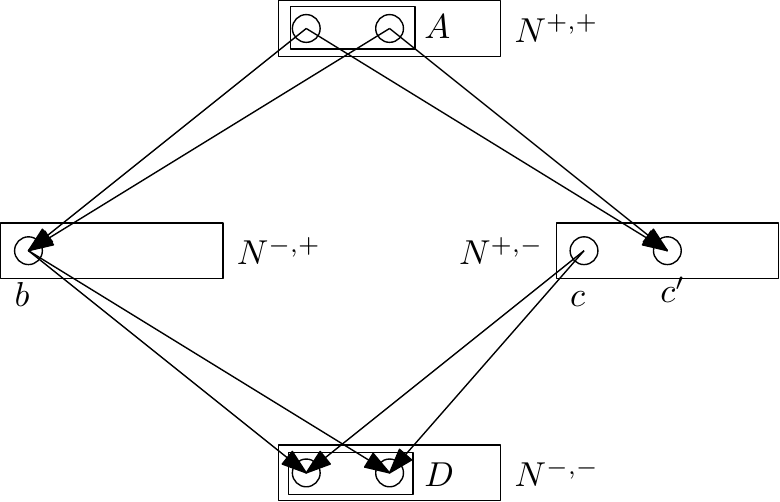}
      \caption{$c' \in N^{+,-} \setminus \{c\}$}
    \label{7configa}
    \end{subfigure}
    \begin{subfigure}{.45\textwidth}
      \includegraphics[scale=0.75]{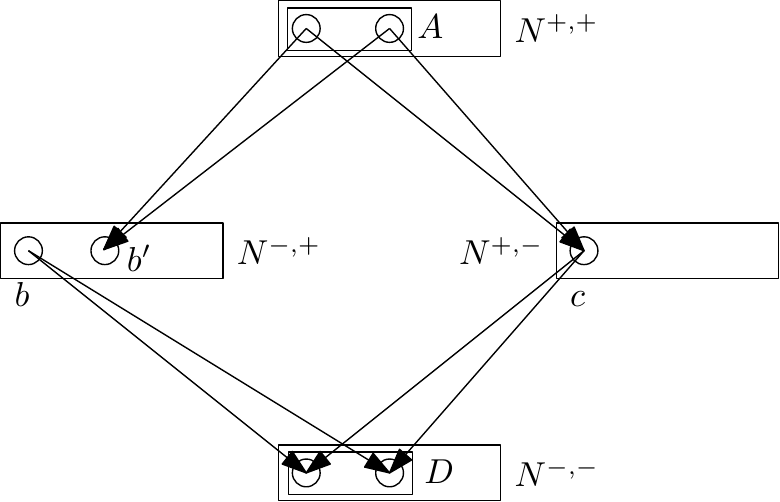}
      \caption{$b' \in N^{-,+} \setminus \{b\}$}
    \label{7configb}
    \end{subfigure}
\caption{Forbidden pairs from the proof of Claim~\ref{confclm}}\label{7config}
}
\end{figure}

\begin{claim}\label{confclm}
  Suppose $b \in N^{-,+}$,  $c \in N^{+,-}$ and $D \subseteq N^{-,-}$ such that 
  $|D| = 2$, and, for every $d \in D$, we have $b \to d$ and $c \to d$.
  \begin{enumerate}[label=(C\arabic*)]
    \item\label{C1} There exists $a' \in N^{+,+}$ such that $b \to a'$ and $c \to a'$.
    \item\label{C2} For every $v \in N^{-,+} \cup N^{+,-}$ such that $v \notin \{b, c\}$,  
      there exists $a \in N^{+,+}$ such that $v \to a$ and there 
      exists  $d \in D$ such that $d \to v$.
  \end{enumerate}
\end{claim}
\begin{proof}
  Call $(A, c')$ a \textit{forbidden pair} if 
  $A$ is a $2$-subset of $N^{+,+}$, $c' \in N^{+,-}\setminus\{c\}$, and, 
  for every $a \in A$, we have $a \to b$ and  $a \to c'$.
  Note that no forbidden pairs can exist because if $(A, c')$ is a forbidden pair, then
  $\{x, c'\} \cup A$, $\{b, c\} \cup D$, $\{y, c\} \cup D$ and $\{b, c'\} \cup A$
  each induce a $\vec{T}_4$, so the set $\{b, c, c'\} \cup A \cup D$ is a linking set (see Figure~\ref{7configa}).
  By similar logic, 
  a pair $(A, b')$ where $A$ is a $2$-subset of $N^{+,+}$, $b' \in N^{-,+} \setminus \{b\}$ 
  and $a \to b'$ and  $a \to c$ for every $a \in A$ cannot exist.
  Therefore, we also call such a pair $(A, b')$ a \textit{forbidden pair} (see Figure~\ref{7configb}).

  We will first show that
  \begin{equation}\label{vtoa}
    \text{for every } v \in N^{-,+} \cup N^{+,-} \text{ there exists } a \in N^{+,+} \text{ such that } v \to a.
  \end{equation}
  Assume the contrary, so $N^-(v) \supseteq N^{+,+}$ for some $v \in N^{-,+} \cup N^{+,-}$.
  Suppose $v \in N^{-,+}$.
  If $v \neq b$, then by Claim~\ref{deg1clm}, there exists 
  $A \subseteq N^-(v) \cap N^-(c) \cap N^{+,+}$ such that $|A| = 2$, 
  and $(A, v)$ is a forbidden pair, a contradiction.
  If $v = b$, then, by Claim~\ref{deg1clm}, for every $c' \in N^{+,-} \setminus \{ c \}$, 
  there exists $A \subseteq N^-(v) \cap N^-(c') \cap N^{+,+}$ such that $|A| = 2$, 
  and $(A, c')$ is a forbidden pair, a contradiction.
  Similar logic leads to a contradiction when $v \in N^{+,-}$, so \eqref{vtoa} holds.

\smallskip

  By \eqref{vtoa}, there exists $a', a'' \in N^{+,+}$ such
  that $b \to a'$ and $c \to a''$.
  To prove \ref{C1}, we need to show that $a' = a''$, so assume the contrary.
  Note that because $|N^{+,-}| = |N^{+,+}| = 3$, Claim~\ref{deg1clm} and \eqref{vtoa}
  imply that the edges between $N^{+,-}$ and $N^{+,+}$ that satisfy the partial order form a matching of size $3$.
  Therefore, because $c \to a''$ and $a'' \neq a'$, there exists $c' \in N^{+,-} \setminus \{c\}$ such that
  $c' \to a'$. Then $(N^{+,+}\setminus\{a'\}, c')$ is a forbidden pair, a contradiction. 

  Now assume that \ref{C2} does not hold.
  With \eqref{vtoa}, this implies that 
  there exists $v \in N^{-,+} \cup N^{+,-}$ such that $v \notin \{b, c\}$,
  and, for every $d \in D$, we have $v \to d$.
  If $v \in N^{-,+}$, then, since $b \to a'$, Claim~\ref{deg1clm} implies that $a' \to v$. 
  This, with Claim~\ref{deg1clm}, violates \ref{C1} with 
  $v$, $c$ and $D$ playing the roles of $b$, $c$, and $D$, respectively.
  Similarly, if $v \in N^{+,-}$, we violate \ref{C1}
  with $b$, $v$ and $D$ playing the roles of $b$, $c$, and $D$, respectively.
\end{proof}

We now select vertices in the following order (see Figure~\ref{8config}).
\begin{itemize}
  \item By Claim~\ref{a1b1clm}, we can select $b_1 \in N^{-,+}$, $c_1 \in N^{+,-}$ and $D \subseteq N^{-,-}$ so that 
    $|D| = 2$ and, for every $d \in D$,  we have $c_1 \to d$ and $b_1 \to d$.
  \item By Claim~\ref{confclm}\ref{C1} we can select $a_1 \in N^{+,+}$ so that $c_1 \to a_1$ and $b_1 \to a_1$.
  \item By Claim~\ref{cycclm}, we can label $\{a_2, a_3\} =  N^{+,+} \setminus \{a_1\}$ so that $a_2 \to a_1$.
  \item By Claims~\ref{deg1clm} and \ref{confclm}\ref{C2}, we can label $\{c_2, c_3\} = N^{+,-} \setminus \{c_1\}$ 
    so that $c_2 \to a_2$ and $c_3  \to a_3$.
  \item By Claim~\ref{confclm}\ref{C2}, we can select $d_3 \in D$ such that $d_3 \to c_3$.
    By Claim~\ref{deg1clm}, this implies that $c_2 \to d_3$.  
    Furthermore, by Claim~\ref{claim:tri} applied to $G[\{a_3, c_3, d_3\}]$, we 
    have $a_3 \to d_3$.
\end{itemize}

\begin{figure}[ht]
\begin{center}
  \includegraphics[scale=.75]{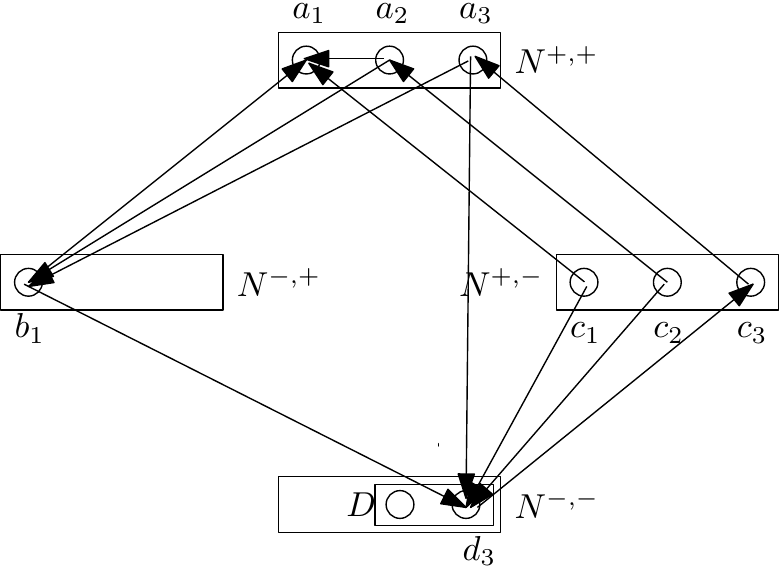}
\end{center}
\caption{The selected vertices at the end of the proof of Lemma~\ref{lemma:main-absorb}.  Note, for $i, j \in [3]$ 
and $i \neq j$ we have $a_i \to c_j$.}\label{8config}
\end{figure}

\bigskip

First note that $N^+(a_2) \supseteq \{a_1, b_1, c_1\}$,
so both $\{a_1, a_2, b_1\}$ and $\{a_1, a_2, c_1\}$ induce transitive triangles.
Since $N^+(y) \supseteq \{a_1, a_2, b_1\}$ and   
$N^+(x) \supseteq \{a_1, a_2, c_1\}$, both $\{y, a_1, a_2, b_1\}$ and   
$\{x, a_1, a_2, c_1\}$  induce copies of $\vec{T}_4$.
Furthermore, 
$N^+(a_3) \supseteq \{b_1, c_1, c_2, d_3\}$ and $N^-(d_3) \supseteq \{a_3, b_1, c_1, c_2\}$, 
so both
$\{a_3, c_1, c_2, d_3\}$, and $\{a_3, b_1, c_2, d_3\}$ induce copies of $\vec{T}_4$.
Therefore, $\{a_1, a_2, a_3, b_1, c_1, c_2, d_3\}$ is a linking set.
This contradiction completes the proof of the lemma.
\end{proof}

\begin{corollary} \label{cor:A(4,2)}
$A(4,2) \le \frac{11}{12}$.
\end{corollary}

\begin{proof}
Let $0< 1/n_0 \ll \eta \ll \varepsilon\ll 1$.
Let $G$ be an oriented graph on $n \ge n_0$ vertices with  $\delta(G) \ge \left(\frac{11}{12} + \varepsilon\right)n$.
Consider any distinct vertices $x$ and $y$ in~$G$.
Let $U := N(x) \cap N(y)$ and note that $|U| \ge 2 \delta(G) - n \ge (10/12 + 2 \varepsilon)n$.
		By Fact~\ref{fact:deg} (with $r=12$, $s = 10$ and $c =  \varepsilon$), 
    we have that 
    \begin{equation*}
      \delta(G[U]) \ge \left(\frac{9}{10} + \frac{12}{10} \varepsilon \right)|U|,
    \end{equation*}
    so, by supersaturation, there exists at least $\eta n^{11}$ tournaments on 
    $11$ vertices in $G[U]$.
    By Lemma~\ref{lemma:main-absorb}, in every such tournament, there
    exists a set $Z$ on $7$ vertices such that $G[\{x\} \cup Z]$ and 
    $G[\{y\} \cup Z]$ both contain a perfect $\vec{T}_4$-tiling.
    Since each such set $Z$ is contained in at most $n^4$ tournaments on $11$ vertices in $G$,
    there are at least $\eta n^7$ such sets $Z$. As $\varepsilon >0$ can be chosen arbitrarily small,
		 $A(4,2) \le \frac{11}{12}$, as required. 
\end{proof}

\section{Concluding remarks and open questions}
In this paper we have asymptotically determined the minimum degree required to force a perfect $\vec{T}_4$-tiling in an oriented graph (Theorem~\ref{thm:T4tiling}). We also obtained  bounds for the general perfect $\vec{T}_k$-tiling problem (Theorem~\ref{inone})
and the perfect fractional $\vec{T}_k$-tiling problem (Theorem~\ref{mainobs}).
In light of Theorem~\ref{mainobs} it would be  interesting to determine whether one can ensure a perfect $\vec{T}_k$-tiling in an oriented graph $G$ of minimum degree $(1 - 1/\trs{k}+o(1))|G|$.
\begin{question}\label{ques} Let $n,k \in \mathbb N$ where $k$ divides $n$ and $k\geq 4$.
Does every $n$-vertex graph with 
$$\delta (G) > (1 - 1/\trs{k}+o(1))n$$
contain a perfect $\vec{T}_k$-tiling?
\end{question}
Note that the $k=4$ case of Question~\ref{ques} is answered in the affirmative by Theorem~\ref{thm:T4tiling}.
If one can show that, for all $k \geq 5$,
$$A(k)\leq 1 - 1/\trs{k}$$
then together with Theorem~\ref{mainobs} and Lemma~\ref{frac+absorb} this would positively answer Question~\ref{ques}. 

For large $k$, Theorem~\ref{mainobs} gives rather close upper and lower bounds on the threshold for perfect fractional $\vec{T}_k$-tiling in oriented graphs (recall that $\trs{k}$ grows exponentially with $k$).
We suspect that  it is possible one can improve on the lower bound in Theorem~\ref{mainobs} (perhaps the upper bound is in fact tight).

It would also be interesting to close the bounds on $\trs{k}$ in Theorem~\ref{boundy}; indeed as discussed in Section~\ref{remarks} this could even lead to improvements on the lower bounds on  $\rr{k}$ and the classical Ramsey numbers $R(k,k)$.
It is also natural to seek structural information on $\vec{T}_k$-free tournaments on $\rr{k}-1$ vertices. When $k=3,4,5,6$, the unique $\vec{T}_k$-free tournament on $\rr{k}-1$ vertices is regular (see~\cite{sanchez1998tournaments}).
This leads to the following question.
\begin{question}\label{ques2} Let $k \geq 3$. Is every $\vec{T}_k$-free tournament on $\rr{k}-1$ vertices a regular tournament?
\end{question}
As noted by a referee, it is not even clear that $\rr{k}$ is even for all $k\geq 3$ (a necessary condition for Question~\ref{ques2} to have an affirmative answer). So this in itself is an interesting question.

Answering Question~\ref{ques2} may also provide insight on the problem (raised in~\cite{treglown2012note}) of determining the minimum \emph{semidegree} that forces an oriented graph to contain a perfect $\vec{T}_k$-tiling. 
Indeed,
given a fixed $k \geq 3$, let $\mathrm{reg}(k)$ denote the size of the largest $\vec{T}_k$-free regular tournament.
Construct an oriented graph $G_{n,k}$ as follows. The vertex set of $G_{n,k}$ consists of a set $A$ of $n/k-1$ vertices and a set $B$ of $(1-1/k)n+1$ vertices; $G_{n,k}[A]$ induces a tournament so that for every vertex in this tournament, its in- and outdegree differs by at most one.
Further $G_{n,k}[B]$ is a blow-up of a
$\vec{T}_k$-free regular tournament $T$ on $\mathrm{reg}(k)$ vertices where the independent sets in $B$ corresponding to  vertices in $T$ are as equally sized as possible (more generally, we could let $G_{n,k}[B]$ be a $\vec{T}_k$-free oriented graph on $|B|$ vertices having the largest possible minimum semidegree; however, we suspect that such an oriented graph will come from the blow-up of a $\vec{T}_k$-free regular tournament $T$ on $\mathrm{reg}(k)$ vertices).
Finally, add all possible edges between $A$ and $B$ in $G_{n,k}$, oriented to ensure that for every vertex $v$ in  $G_{n,k}$, $d^+_{G_{n,k}}(v)$ and $d^-_{G_{n,k}}(v)$ are as close as possible.
Notice that every copy of $\vec{T}_k$ in $G_{n,k}$ must use at least one vertex from $A$; thus as $|A|=n/k-1$, $G_{n,k}$ does not contain a perfect $\vec{T}_k$-tiling.
Further, certainly $\delta ^0 (G_{n,k}) \geq \left (\frac{1}{2}- \frac{(k-1)}{2k \cdot \mathrm{reg} (k)} -o(1) \right )n.$

Note that $G_{n,k}$ is a generalization of the example given in~\cite[Proposition~6]{treglown2012note} (which deals with the case when $k=3$). Further,
in~\cite{balogh2017transitive} it was proven that $G_{n,3}$ is an extremal example for the minimum semidgree problem for perfect $\vec{T}_3$-tilings. That is,
all sufficiently large oriented graphs on $n$ vertices whose minimum semidegree is above that of $G_{n,k}$ contains a perfect $\vec{T}_3$-tiling.
Thus, it is natural to ask the following question.
\begin{question}\label{ques3} Let $k,n \geq 3$ so that $k$ divides $n$. Does every oriented graph $G$ on $n$ vertices  with 
$$\delta ^0 (G) > \left (\frac{1}{2}- \frac{(k-1)}{2k \cdot \mathrm{reg} (k)} +o(1) \right )n$$
contain a perfect $\vec{T}_k$-tiling?
\end{question}
\section{Acknowledgments}

This project began during the ``Recent Advances in Extremal Combinatorics Workshop'' at the Tsinghua Sanya International Mathematics Forum, May 22-26, 2017.  We thank the organizers of this conference for the stimulating work environment.  

We also thank the referees for their helpful and careful reviews.

\newpage

\section{Appendix: Tournaments on $12$ vertices that do not have a perfect $\vec{T}_4$-tiling}\label{app:12_no_tt4}

In Figure~\ref{fig:12_no_tt4}, we list $43$ tournaments on $12$ vertices that do not have a perfect $\vec{T}_4$-tiling.   
This is an exhaustive list (up to isomorphism) of such tournaments.
The $66 = \binom{12}{2}$ numbers in each line represent the entries in the upper triangle of the $n \times n$ matrix 
$\left(a_{i,j}\right)$ where 
$a_{i,j}=1$ if the edge incident to $v_i$ and $v_j$ is directed from $v_i$ to $v_j$ 
and $a_{i,j}=0$ otherwise. These entries are listed in the following order 
\begin{equation*}
  a_{1,2}a_{1,3} \dotsm a_{1,12}a_{2,3}a_{2,4} \dotsm a_{2,12} \dotsc a_{10,11}a_{10,12}a_{11,12}.
\end{equation*}
(This is the default output format for the program \texttt{gentourng} which 
is a program that is distributed with \texttt{nauty} and \texttt{Traces} \cite{nauty} that can be used to generate 
all small tournaments.)

\begin{figure}[ht]
  {\tiny \centering 
    \texttt{
      \\
      110011001001111001001111010110101101111101010111101110101100111111\\
      111011111111010101100111010100110010101111001110110110101101111111\\
      111011000001010110110111011010110101011101100111011111011111110111\\
      110011010001110100100111011001101110101110101110011111111111110101\\
      111011100001010110110111001010110101011101100111011111011111110111\\
      110011101111110010100111001010101110011110110101100110101101111111\\
      111100010001101010100110110011111011011111010111001101111110111101\\
      101011000101111000100110101000110110111110101111110111011111110111\\
      110011101001110110010111001100101101011101010111001111111111110111\\
      101011001001111001001110111111110101101101010111101110101100111111\\
      111100101001101011001110101111111100101101100110101110101111110111\\
      101010011001111010000110100101111010111110110111101111111010110111\\
      111111111111010101100111010100110010101111001110110110101101111111\\
      110100110001101101001111110010110011001110101111011101111110111101\\
      110101110001101010010111100100110110011101011110110111011111111101\\
      110010101001110011001111110010101101011101111101100110101111110111\\
      101010110011110101111110010100111100101101100110101110101111110111\\
      101010111111110101000110010101111100101101101110100110111110110111\\
      101010110011110101000110010100111100111101101110111110101110111101\\
      101010110011110101001110010101111100111101101110101110111111111111\\
      101010110001110101001110010101111100111101101110101110111111111111\\
      101010110001110101000110010101111100111101101110101110111111111111\\
      101010110001110101000110010100111100111101101110101110111111111111\\
      101010110001110101000110010101111100101101101110101110111111111111\\
      101010110001110101001110010100111100101101101110101110111111111111\\
      101010110001110101000110010100111100101101101110101110111111111111\\
      101010110001110101000110010100111100111101100110101110111111111111\\
      101010110001110101000110010101111100101101100110101110111111111111\\
      101010110001110101001110010100111100101101100110101110111111111111\\
      101010110011110101000110010100111100101101100110101110111111111111\\
      101010110001110101001110010100111100101101101110100110111111111111\\
      101010110001110101000110010100111100101101100110101110111111111111\\
      101010110001110101000110010100111100101101101110100110111111111111\\
      101010110001110101000110010100111100111101100110100110111111111111\\
      101010110001110101000110010101111100101101100110100110111111111111\\
      101010110001110101001110010100111100101101100110100110111111111111\\
      101010110011110101000110010100111100101101100110100110111111111111\\
      101010110011110101000110010100111100101101100110101110101111111111\\
      101010110001110101001110010100111100101101101110100110101111111111\\
      101010110011110101000110010100111100101101101110100110101111111111\\
      101010110011110101000110010101111100101101100110100110101111111111\\
      101011100101111000100110101000110110111110101111110111011111110111\\
      101010110011110101000111010100111100111101101110111110101110111101\\
    }
}
\caption{$43$ tournaments on $12$ vertices that do not have a perfect $\vec{T}_4$-tiling.}
  \label{fig:12_no_tt4}
\end{figure}

\end{document}